%% file: rep.tex
\documentclass{article}

\usepackage[T1]{fontenc}     
\usepackage{amsthm}  
\usepackage{amsmath}
\usepackage{amsfonts}
\usepackage{tikz}
\usepackage[capitalize, noabbrev]{cleveref}
\usepackage{tikz-cd}
\usepackage{amssymb}
\usepackage{filecontents}
\usepackage{graphicx}
\usepackage{subcaption}
\usepackage{float}

\newtheorem{theo}{Theorem}[section]

\newtheorem{cor}[theo]{Corollary}
\newtheorem{prop}[theo]{Proposition}
\newtheorem{lemma}[theo]{Lemma}
\theoremstyle{remark}{}

\theoremstyle{definition}

\newtheorem{defn}[theo]{Definition}

\newcommand{\C}{\mathbb C}
\newcommand{\R}{\mathbb R}
\newcommand{\id}{\mathrm{id}}
\newcommand{\Id}{\mathrm{Id}}
\newcommand{\homeo}[1]{\mathrm{Homeo}^+ (#1)}

\newcommand{\SNS}{\mathcal {S}_\mathrm{ns}}
\newcommand{\PSL}[1]{\mathrm{PSL}_2 (#1)}
\newcommand{\SL}[1]{\mathrm{SL}_2 (#1)}
\newcommand{\Mod }{\mathrm{Mod}}
\newcommand{\CP}{\mathbb{CP}^1}
\newcommand{\hol}{\mathrm{hol}}

\newcommand{\Autp}[1]{\mathrm{Aut}^+ (#1)}
\newcommand{\Tr}{\mathrm{Tr}}

\author{Thomas Le Fils}
\date{}

\begin{document}

\title{Pentagon representations and complex projective structures on closed surfaces}

\maketitle

\begin{abstract}
We define a class of representations of the fundamental group of a closed surface of genus $2$ to $\PSL \C$: \textit{the pentagon representations}. We show that they are exactly the non-elementary $\PSL \C$-representations of surface groups that do not admit a Schottky decomposition, \textit{i.e.} a pants decomposition such that the restriction of the representation to each pair of pants is an isomorphism onto a Schottky group. 
In doing so, we exhibit a gap in the proof of Gallo, Kapovich and Marden that every non-elementary representation of a surface group to $\PSL \C$ is the holonomy of a projective structure, possibly with one branched point of order $2$.
We show that pentagon representations arise as such holonomies and repair their proof.

\end{abstract}
\section{Introduction}

Denote by $\Sigma_{g,n}$ an oriented compact surface of genus $g$ with $n$ boundary components, and by $\Gamma_{g,n}$ a fundamental group of $\Sigma_{g,n}$ for all $g,n\geqslant 0$. For simplicity, we denote $\Sigma_{g,0}$ by $\Sigma_g$ and $\Gamma_{g,0}$ by $\Gamma_g$.

A complex projective structure on $\Sigma_g$ is a $(G,X)$ structure with $G= \PSL \C$ and $X = \CP$, that is the datum of an atlas of charts with values in $\CP$, whose transition maps are restrictions of M\"obius transformations. We can also allow branched points in the definition and get the notion of branched projective structure, see \cite[Section~1.4]{GKM} for a concise definition. We denote by $\mathcal P(\Sigma_g)$, resp. $\mathcal P_b(\Sigma_g)$, the set of unbranched complex projective structures, resp. the set of projective structures with a single branched point, of order $2$. The datum of such a structure on $\Sigma_g$ gives rise to a holonomy map, well-defined up to conjugacy (see \cite[Chapter 3]{Thurston} for more information on $(G,X)$-structures). Hence we have a map:

\[\hol : \mathcal P(\Sigma_g)\sqcup \mathcal P_b(\Sigma_g)\to  \mathrm{Hom}(\Gamma_g, \PSL \C)/\PSL \C.\]
This map establishes a relationship between $\PSL \C$-representations of surface groups and projective structures. The study of this relationship has a long history. There is a natural complex structure on $\mathcal P(S)$, induced by its identification with the quadratic forms (see \cite{Goldmanwhatis} for example). Hejhal, Earl and Hubbard \cite{hejhal, earle, hubbard} showed that the map $\hol_{|\mathcal P(\Sigma_g)}$ is a local biholomorphism. However, it is known that $\hol_{|\mathcal P(\Sigma_g)}$ is neither injective nor a covering map. We refer to \cite{Dumas} for more information about projective structures.

The question of finding which representations arise as the holonomy of a complex projective structure has been open for a long time. Poincar\' e himself asked it in the case where $\Sigma$ is a punctured sphere, see \cite[Paragraph 4]{poincare}. Very recently, Gupta announced an answer for every punctured surface~\cite{gupta2019monodromy}. In \cite{GKM}, Gallo, Kapovich and Marden provided a complete answer for closed surfaces. They showed that the image of $\hol$ is the set of non-elementary representations.
The main part of the theorem is the proof that every non-elementary representation is in the image of $\hol$.

The strategy of \cite{GKM} consists in first proving that every non-elementary representation admits a \emph{Schottky decomposition} in the following sense.

\begin{defn}
A \emph{Schottky decomposition} for a representation $\rho : \Gamma_g\to \PSL\C$ is a pants decomposition of $\Sigma_g= \cup P_i$, such that for all $i$, the restriction $\rho_{|P_i}~:~\pi_1(P_i)\to \PSL\C$ is an isomorphism onto a Schottky group.
\end{defn}

Once such a decomposition is found, the authors put a projective structure on each pair of pants $P_i$, whose holonomy is given by the restriction $\rho_{|P_i}$. Then they glue the pants together with cylinders. It might be required to add a branched point of order $2$ in one of the pair of pants in order to make all the gluings possible.

Perhaps one main contribution of the present article is to exhibit a gap
in the proof by Gallo, Kapovich and Marden of the existence of a Schottky decomposition for every non-elementary representation. We will establish later (see \cref{theo_un} below), essentially along their proof, that such a decomposition exists provided $g\geqslant 3$. However, in genus 2, the hyperelliptic involution yields counterexamples that we now introduce.

Let us recall that the mapping class group $\Mod(\Sigma_2)$ of $\Sigma_2$ has its center generated by the hyperelliptic involution. Let $\varphi\in \mathrm{Homeo}^+(\Sigma_2)$ be one of its representatives. The orbifold fundamental group $\Gamma$ of $\Sigma_2/\varphi$ has the following presentation: $$\Gamma = \langle q_1, q_2, q_3, q_4, q_5, q_6 \mid q_i^2 = 1, q_1q_2\ldots q_6 = 1 \rangle.$$
The group $\Gamma_2$ is naturally an index two subgroup of $\Gamma$ (see \cref{Pentagon}).

\begin{defn}
A \emph{pentagon representation} is a non-elementary representation that is the restriction of a representation $\rho : \Gamma\to \PSL \C$ such that $\rho(q_i) = \id$ for exactly one $1\leqslant i\leqslant 6$.
\end{defn}
This definition is reminiscent of the hourglass representations considered in~\cite{MarcheWolff}.

\begin{prop}
A pentagon representation does not admit a Schottky decomposition.
\end{prop}

As a first step into reparing the proof of \cite{GKM}, we prove that there exists a Schottky decomposition for every other non-elementary $\rho$.
\begin{theo}\label{theo_un}
A representation $\rho : \Gamma_g\to \PSL\C$ admits a Schottky decomposition if and only if $\rho$ is non-elementary and is not a pentagon representation.
\end{theo}

In fact we find a Schottky decomposition with $g$ pairs of pants glued to themselves in the non-pentagon case. This allows us to use the second part of \cite{GKM}, and hence fix the gap of the proof in this case.

As a corollary of this characterization of the pentagon representations, we will see that even if they do not admit a Schottky decomposition, they still have a \emph{loxodromic decomposition}.
\begin{cor}\label{cor}
If $\rho : \Gamma_g\to \PSL \C$ is non-elementary, then there exists a pants decomposition of $\Sigma_g$ whose boundary curves are taken by $\rho$ to loxodromic isometries.
\end{cor}
We hence also repair the proof of this corollary which was used for example in \cite[Proposition~1]{tan}.

We show that the pentagon representations have odd Stiefel-Whitney class, thus they cannot be in $\hol(\mathcal P(\Sigma_2))$ (see \cite[Corollary~11.2.3]{GKM}).
It remains to understand whether they are in $\hol(\mathcal P_b(\Sigma_2))$ or not. We answer positively:

\begin{theo}\label{PentagonAreProjective}
A pentagon representation is the holonomy of a branched projective structure with exactly one branched point, which is of order $2$.
\end{theo}

\cref{PentagonAreProjective} somehow complements a recent theorem of Baba \cite{BabaSchottky}, which states that every unbranched complex projective structure is obtained by gluing Schottky pants as in \cite{GKM}. Indeed the analogous theorem for $\mathcal P_b(\Sigma_2)$ cannot hold for the pentagon representations have no Schottky decomposition and yet are in $\hol (\mathcal P_b(\Sigma_2))$.

In view of possible generalizations, we provide some
alternative (and maybe simpler) proofs of some intermediate
results of [6], that we use to prove \cref{theo_un}.
In particular, we give a new proof of the existence of a special handle in $\Sigma_{g,n}$: a subsurface which is a punctured torus, whose fundamental group is generated by two elements $a$ and $b$ such that $\rho(a)$ and $\rho(b)$ are loxodromic without a common  fixed point, where $\rho : \Gamma_{g,n}\to \PSL \C$ is non-elementary and $g\geqslant 1$. Our approach begins by finding a handle on which the restriction of $\rho$ is non-elementary. This leads us to study the non-elementary representations of the punctured torus, up to the action of the mapping class group. This study is reminiscent of the work of Goldman in \cite{Goldman}. In particular we prove the following.

\begin{theo}\label{bonne_anse}
If $\rho : \Gamma_{1,1}\to \PSL \C$ is non-elementary, then there exists simple loops $a$ and $b$ generating $\Gamma_{1,1}$ such that $\rho(a)$ and $\rho(b)$ are loxodromic.
\end{theo}

It follows that:

\begin{cor}
If $g\geqslant 1$, and $\rho : \Gamma_{g,n}\to \PSL\C$ is non-elementary, then there exists a simple curve $\gamma\in \Gamma_{g,n}$ such that $\rho(\gamma)$ is loxodromic.
\end{cor}

Note that this theorem does not hold for $g=0$. Indeed, Baba observed that some $\PSL \R$-representations of $\Gamma_{0,4}$ studied in \cite{BenedettoGoldman} are non-elementary and send every simple curve to an elliptic isometry. Later in \cite{DeroinTholozan}, Deroin and Tholozan exhibited a class of non-elementary representations of $\Gamma_{0,n}$ into $\PSL \R$ that send every simple closed curve to a non-hyperbolic isometry for every $n\geqslant 4$.

Let us now describe the organization of the paper. In \cref{section_handle}, we recall general facts about curves on surfaces before proving the existence of a special handle given a non-elementary $\rho$. We proceed as described above and prove \cref{bonne_anse}. Then we study pentagon representions in \cref{Pentagon}. \cref{section_schottky} is devoted to the proof of \cref{theo_un}. Finally, we show that the pentagon representations are in $\hol(\mathcal P_b(\Sigma_2))$ in \cref{section_projective}.

\subsubsection*{Acknowledgements}
I would like to thank my advisor M. Wolff for his help and guidance. I am also grateful to B. Deroin for his comments.
\section{Special handle}\label{section_handle}

\subsection{Reminder of curves on surfaces}

If $\gamma$ and $\delta$ are elements of $\Gamma_{g,n}$, the loop $\gamma\delta$ is the path following $\delta$ first and then~$\gamma$. We also define the commutator of $\gamma$ and $\delta$ to be $[\gamma, \delta] = \delta^{-1}\gamma^{-1}\delta\gamma$.

We fix a system of \emph{standard generators} $a_1, b_1, \ldots a_g, b_g, c_1, \ldots c_n$ of $\Gamma_{g,n}$, in the same way as \cref{syst_gen} for $(g,n)=(2,1)$.

\begin{figure}[h]
\centering
\def\svgwidth{0.8\textwidth}
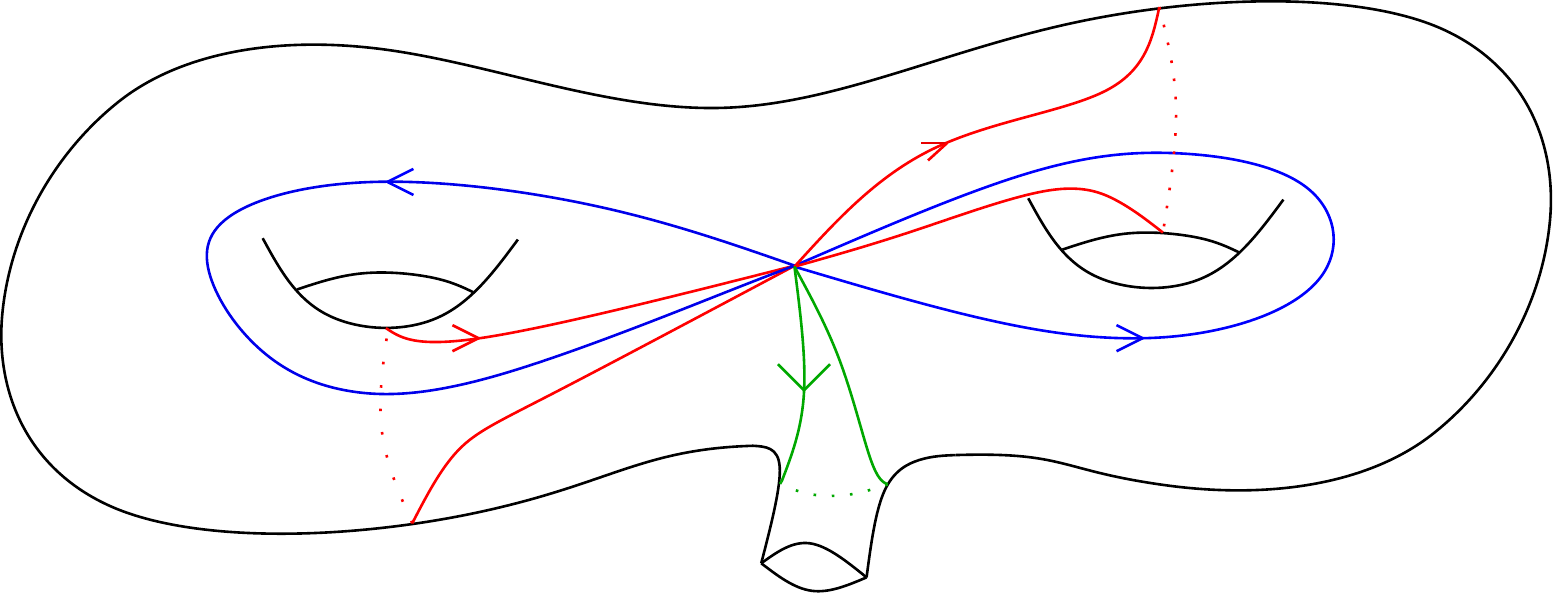
\caption{Standard generators for $\Gamma_{2, 1}$.}
\label{syst_gen}
\end{figure}

Denote by $\SNS$ the set of $\gamma\in \Gamma_{g,n}$ such that $\gamma$ is freely homotopic to an essential non-separating simple closed curve.
If $f\in \homeo {\Sigma_g, *}$, \textit{i.e.} $f$ fixes the base point $*$ of $\Gamma_g$, we denote by $f_*$ the automorphism of $\Gamma_g$ induced by $f$.
\begin{lemma}\label{lemme_courbe_signe}
Given $\gamma$ and $\delta$ in $\SNS$, there exists $f\in \homeo {\Sigma_{g,n}, *}$ such that $f_*(\gamma) \in \{\delta, \delta^{-1}\}$.
\end{lemma}

\begin{proof}
This is a consequence of classification of compact surfaces. See \cite[Section~1.3]{FarbMargalit}. Note that if $g \geqslant 1$, and $*\notin \partial \Sigma_{g,n}$, we can even choose $f$ such that $f_*(\gamma) = \delta$.
\end{proof}

Now assume that $g \geqslant 1$, and that $*\notin \partial \Sigma_{g,n}$.

\begin{lemma}\label{poigne}
Let $\gamma,\delta\in\SNS$ be such that they have representatives that cross only at $*$ as in \cref{pic_handle}.

\begin{figure}[h]
\centering
\def\svgwidth{0.25\textwidth}
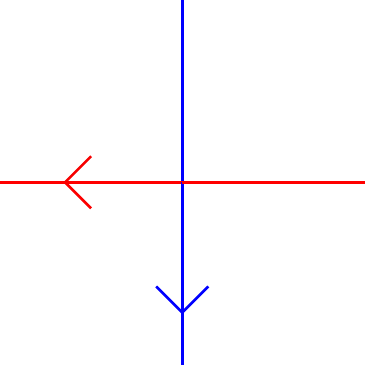
\caption{Handle.}
\label{pic_handle}
\end{figure}
There exists $f\in \homeo{\Sigma_{g,n}, *}$ such that $f_*(\gamma) = a_1$ and $f_*(\delta) = b_1$.
\end{lemma}

\begin{proof}
This is also a consequence of the classification of surfaces, see \cite[Section~1.3.3]{FarbMargalit}. 
\end{proof}

\subsection{Existence of a non-elementary handle}

Let us recall that $\PSL \C$ is the group of isometries of $\mathbb H^3$. Its action on~$\mathbb H^3$ extends to the Gromov boundary $\partial \mathbb H^3 \simeq \CP$ where it acts by M\"obius transformations. We say that a homomorphism $\rho : G \to \PSL \C$ is elementary if there exists $z\in \mathbb H^3\cup \partial \mathbb H^3$ such that $\rho(G)\cdot z$ is finite. We have, see for example \cite[Chapter 5]{Ratcliffe}, the following characterization.

\begin{prop}[\cite{Ratcliffe}]
A homomorphism $\rho : G\to \PSL \C$ is elementary if and only if there exists a set $S\subset \mathbb H^3\cup \partial \mathbb H^3$ containing $1$ or $2$ points such that $\rho(G)$ stabilizes $S$.
\end{prop}

Suppose $g\geqslant 1$ and $(g,n)\neq (1,0)$. Fix a non-elementary $\rho : \Gamma_{g,n}\to \PSL \C$. We also assume that the base point of $\Gamma_{g,n}$ is not on the boundary $\partial \Sigma_{g,n}$.
\begin{defn}
A \emph{non-elementary handle} is a subsurface of $S\subset \Sigma_{g,n}$ which is a punctured torus, such that the restriction $\rho_{|S} : \Gamma_{1,1}\to \PSL\C$ is non-elementary.
\end{defn}

Let us recall three lemmas from \cite{GKM}.
\begin{lemma}
Suppose $\alpha,\beta\in \PSL \C$.
\begin{itemize}
\item If $\alpha$ is loxodromic, and $\beta$ does not send its attractive fixed point to its repulsive one (resp. its repulsive one to its attractive one), then there exists $K>0$ such that $\alpha^k \beta$ is loxodromic for all $k \geqslant K$ (resp. for all $k\leqslant -K$). Moreover, the trace of $\alpha^k\beta$ can be made arbitrarily large.
\item If $\alpha$ is parabolic and $\beta$ does not fix its fixed point, then $\alpha^k\beta$ is loxodromic for $|k|$ large enough.

\end{itemize}
\end{lemma}

\begin{proof}
This is a trace computation, see \cite[Lemma~2.2.1]{GKM}.
\end{proof}

\begin{lemma}
If there exists $\gamma\in \SNS$ such that $\rho(\gamma)$ is loxodromic or parabolic, then there is a handle in $\Sigma_{g,n}$ on which the restriction of $\rho$ is non-elementary.
\end{lemma}

\begin{proof}
We can assume that $\gamma = a_1$. If $\rho(b_1)$  stabilizes the set of fixed points of $\rho(\gamma)$, take $c\in \{a_2^{-1}, b_2, \ldots ,b_g, c_1,\ldots ,c_n\}$ such that $\rho(c)$ does not, and apply a homeomorphism so that $(a_1, cb_1)$ becomes $(a_1, b_1)$, which exists by \cref{poigne}.

If $\rho(a_1)$ is parabolic, then the group generated by $\rho (a_1)$ and $\rho(b_1)$ is non-elementary.  
This is also the case if $\rho(a_1)$ is loxodromic and $\rho(b_2)$ does not share a fixed point with $\rho(a_1)$. Otherwise denote by $p$ this common fixed point. Choose $c\in \{a_2^{-1}, b_2, \ldots ,b_g, c_1,\ldots ,c_n\}$ such that $\rho(c)$ does not fix $p$.
Then $(ca_1^{-1}, a_1^kb_1)$ is a such handle for some $k$. Indeed there exists $K>0$ such that $\rho(a_1^k b_1)$ is loxodromic for $k\geqslant K$ or $k\leqslant -K$. Moreover the fixed point of $\rho(a_1^k b_1)$ which is not $p$ is different for any two $k$ in that range: if $\rho(a_1^k b_1)(q) = \rho(a_1^{k+m}b_1)(q) = q$, then $\rho(a_1)(q) = \rho(b_1)(q) = q$ and $q=p$. Thus we can take a $k$ in that range that does not share a fixed point with $\rho(ca_1^{-1})$.
\end{proof}

\begin{lemma}\label{axes_produit_elliptique}
If $\alpha$ and $\beta$ are elliptic with different axes and $\alpha\beta$ is elliptic, then:

\begin{enumerate}
\item The axes of $\alpha$ and of $\beta$ lie in a plane $P$.
\item If they are disjoint, they are orthogonal to a plane.
\item The axis of $\alpha\beta$ is not contained in $P$.
\end{enumerate}
\end{lemma}

\begin{proof}
Decompose $\alpha$ and $\beta$ as: $\alpha = s_{\ell_2} s_{\ell_1}$ and $\beta = s_{\ell_3}s_{\ell_2}$ where $s_{\ell_i}$ is the elliptic involution with axis $\ell_i$. See \cite[Lemma~3.4.1,~3.4.3]{GKM}.
\end{proof}

\begin{prop}
There exists a non-elementary handle on $\Sigma_{g,n}$.
\end{prop}

\begin{proof}
We may assume $g \geqslant 2$ or $n\geqslant 2$ because the proposition is obvious otherwise.
By contradiction suppose that the restriction of $\rho$ to any handle on $\Sigma_{g,n}$ is elementary. For every $\gamma\in \SNS$, $\rho(\gamma)$ is elliptic or the identity.
As before, we can assume that $\rho(a_1)$ is not the identity, that $\rho(b_1)$ is not the identity and does not have the same axis as $\rho(a_1)$.
Their axes cross since $a_1$ and $b_1$ bound a handle on which $\rho$ is elementary.
Pick $c'\in \{a_2, b_2^{-1}, \ldots ,b_g^{-1}, c_1^{-1}, \ldots, c_n^{-1}\}$ so that $\rho(c')$ does not fix the common fixed point of $\rho(a_1)$ and $\rho(b_1)$. Let $c = c'b_1^{-1}a_1^{-1}$, so that any two out of $a_1,b_1,c$ form a handle.

The axes of $\rho(a_1), \rho(b_1)$ and $\rho(c)$ form a triangle $T$. 

Since $\rho(ca_1b_1) = \rho(c')$ is elliptic, the axis of $\rho(c)$ is coplanar with the axis of $\rho(a_1b_1)$. The only plane that contains both the axis of $\rho(c)$ and the common fixed point of $\rho(a_1)$ and $\rho(b_1)$ is the one spanned by $T$. This is a contradition by \cref{axes_produit_elliptique}: the axis of $\rho(a_1b_1)$ is not coplanar with both the axis of $\rho(a_1)$ and the axis of $\rho(b_1)$.
\end{proof}

\subsection{Non-elementary representations of the punctured torus}

The aim of this subsection is to prove \cref{bonne_anse}.
Let $\rho : \Gamma_{1,1}\to \PSL\C$ be a non-elementary representation. Our strategy is to precompose $\rho$ by automorphisms of $\Gamma_{1,1}$ induced by Dehn twists along curves homotopic to the standard generators $a$ and $b$ of $\Gamma_{1,1}$. The automorphisms we consider are defined by $(\varphi_k(a), \varphi_k(b)) = (b^ka, b)$  and $(\psi_k(a), \psi_k(b)) = (a,a^kb)$ for $k\in \mathbb Z$.
\begin{lemma}
If there exists $\gamma\in \SNS$ such that $\rho(\gamma)$ is parabolic or loxodromic, then \cref{bonne_anse} holds.
\end{lemma}

\begin{proof}
We may assume that $a \in \{\gamma, \gamma^{-1}\}$ by \cref{lemme_courbe_signe}.
If $\rho(a)$ is loxodromic, then so is $\rho(a^k b)$ for some $k$. Apply the Dehn twist that changes $(a,b)$ to $(a, a^kb)$. 
If $\rho(a)$ is parabolic, then $\rho(a^kb)$ is loxodromic if $|k|$ is large enough. We change $(a,b)$ to $(a, a^kb)$ with a Dehn twist, and return to the previous case.
\end{proof}

\subsubsection{Existence of a loxodromic}

We now show that it is not possible for $\rho$ to send every $\gamma\in \SNS$ to an elliptic element or to the identity. Assume that is does by contradiction.

\begin{lemma}
The representation $\rho$ has some conjugate into $\PSL\R$: it preserves a plane.
\end{lemma}

\begin{proof}
The isometries $\rho(a)$ and $\rho(b)$ are not the identity since $\rho$ is non-elementary, and their axes do not cross. Since $ab\in \SNS$, it follows from \cref{axes_produit_elliptique} that these axes are orthogonal to a plane.
\end{proof}

We can now assume that $\rho(\Gamma_{1,1}) \subset \PSL\R$.

\begin{lemma}
There exists $N\geqslant 1$ such that for all $\gamma\in \SNS$, $\rho(\gamma)^N = \id$.
\end{lemma}

Here we denoted by $\id$ the element $\pm \Id$ of $\PSL \C$.

\begin{proof}
Note that every $\rho(\gamma)$ has finite order for $\gamma\in \SNS$. Indeed, if not, we can suppose that $\rho(a)$ has infinite order. But then $a^nb\in \SNS$ has loxodromic image for some $n$. Indeed write $\rho(a^n) = s_{\ell_n}s_{\ell'}$ and $\rho(b) = s_{\ell'}s_{\ell''}$ as products of reflections across geodesics. If $\rho(a)$ had infinite order, then we could take $n\geqslant 1$ such that $\ell_n$ does not cross $\ell''$. But then $\rho(a^nb)$ would be loxodromic. 
\begin{figure}[H]
\centering
\def\svgwidth{0.37\textwidth}
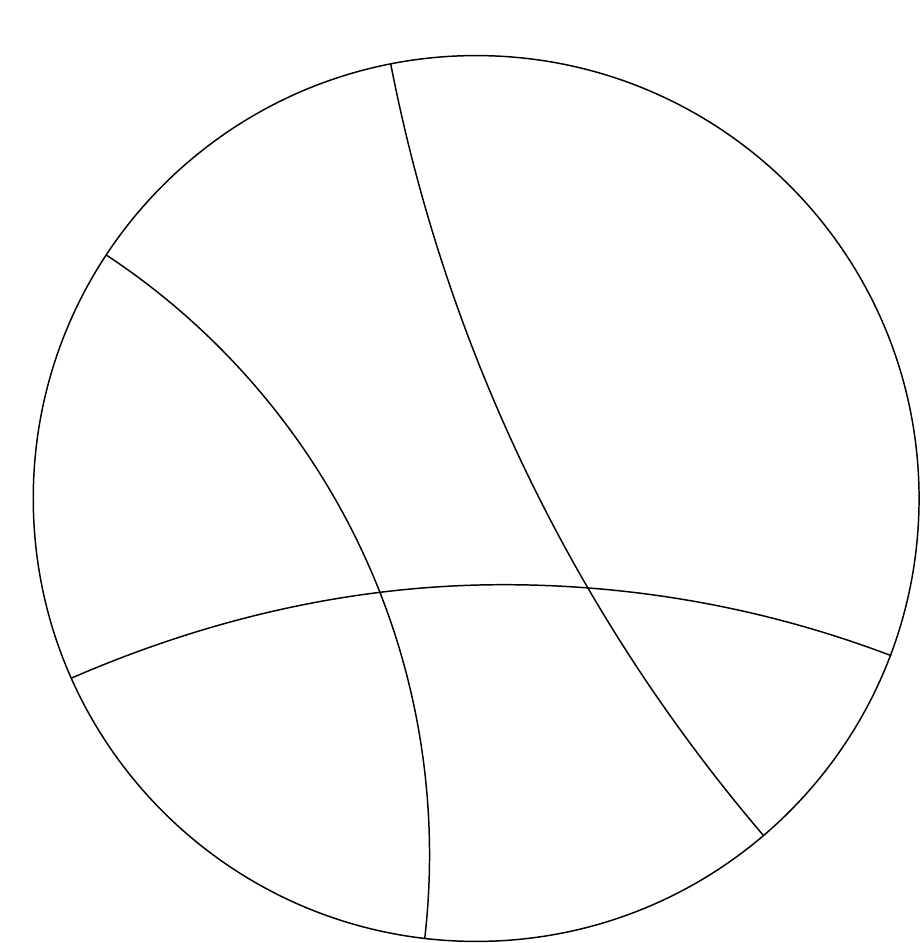
\caption{Product of reflections.}
\label{refl}
\end{figure}
This can also be seen with a computation.

By Selberg's lemma, there is a  torsion free subgroup $\Lambda$ of $\rho(\Gamma_{1,1})$ of finite index $n$. Let $\gamma\in \SNS$. Two of the cosets $\Lambda, \rho(\gamma)\Lambda, \ldots \rho(\gamma)^n\Lambda$ are equal. Thus $\rho(\gamma)^j \in \Lambda$ for a $j\leqslant n$. This implies that $\rho(\gamma)^j = \id$ and it thus suffices to take $N=n!$.
\end{proof}

We get a contradiction from this lemma.
\begin{lemma}
We can increase the order of $\rho(a)$ or $\rho(b)$ by applying a Dehn twist.
\end{lemma}

\begin{proof}
Write $\rho(a) = s_{\ell_2}s_{\ell_1}$ and $\rho(b) = s_{\ell_3}s_{\ell_2}$ as products of reflections. Since $\rho(ba)$ is elliptic, the lines $\ell_1$, $\ell_2$ and $\ell_3$ form a triangle, see \cref{tritri}.

\begin{figure}[h]
\centering
\def\svgwidth{0.45\textwidth}
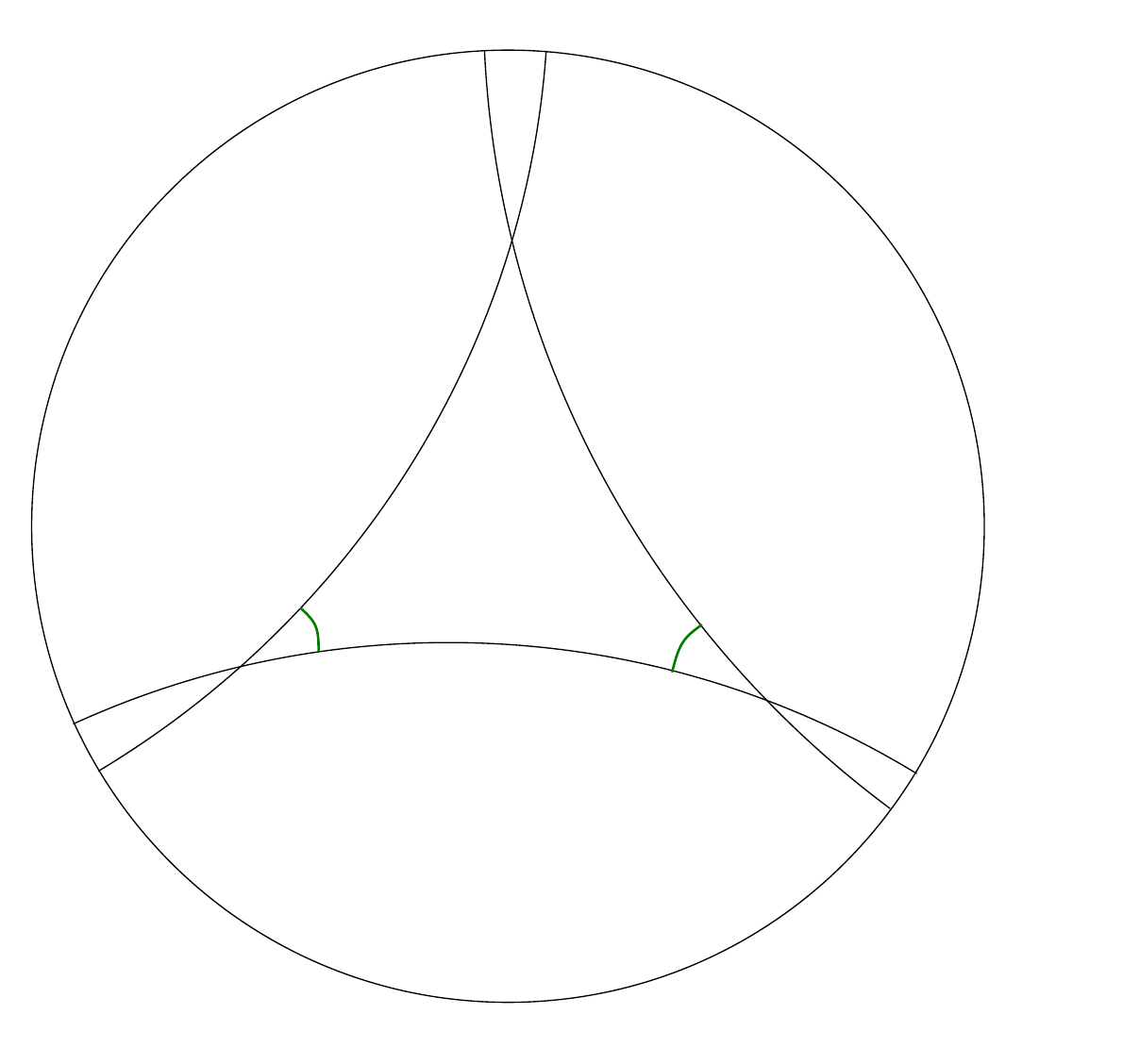
\caption{Triangle $(\ell_1, \ell_2, \ell_3)$.}
\label{tritri}
\end{figure}

Suppose $\beta \leqslant \alpha$. The fixed points of $\rho(b^{-1}a)$ and of $\rho(ba)$ are separated by the line $\ell_2$. Indeed if not, one could contruct a triangle with total angle greater than~$\pi$.

If $\ell_3$ intersects $\ell_1$ in the part of the $\mathbb H^2\setminus \ell_2$ that does not contain the projection of the fixed point of $\rho(b)$ to $\ell_1$, then the order of $\rho(ba)$ is less than the order of $\rho(a)$. Change the handle to $(ba, b)$. If not, change the handle to $(b^{-1}a, b)$.

The case where $\alpha \leqslant \beta$ is symmetric.
\end{proof}

\subsubsection{Special handle}

\begin{defn}
A \emph{special handle} is a non-elementary handle with standard generators $a,b$ such that $\rho(a)$ and $\rho(b)$ are loxodromic, and such that $\rho(a)$ does not send a fixed point of $\rho(b)$ to the other.
\end{defn}

Thanks to \cref{bonne_anse}, we can suppose that the standard generators $a$ and $b$ are sent by $\rho$ to loxodromic isometries. We now modify this handle to get a special one. Our proof relies on the following results of \cite[Lemma 2.2.1, 2.2.2, 2.2.3]{GKM}, obtained by trace computations.

\begin{lemma}[\cite{GKM}]\label{echanger}
\begin{enumerate}
\item If $\xi$ interchanges the fixed points of $ \beta^k \alpha$ and of $\beta^{k+1} \alpha$ which are loxodromic, then $\xi \alpha \xi = \alpha^{-1}$ and $\xi \beta \xi = \alpha^{-1} \beta^{-1} \alpha$.

\item Suppose $\alpha$ and $\beta$ are loxodromic and $\alpha$ sends a fixed point of $\beta$ to the other. Then $\beta$ sends a fixed point of $\alpha$ to the other if and only if $ \Tr^2(\alpha) = \Tr^2(\beta)$.

\item If $\alpha$ and $\beta$ are loxodromic, the fixed points of $\alpha^m\beta$ converge to $p^*$ and $\beta(p_*)$) (resp. $p_*$ and $\beta^{-1}(p_*)$) when $m\to\infty$ (resp. $m\to -\infty$) where $p^*$ is the attractive fixed point of $\alpha$ and $p_*$ its repulsive one.
\end{enumerate}
\end{lemma}

Later we will need a little bit more than the existence of a special handle, as proved in \cite{GKM}.
\begin{prop}\label{arranger_anse}
There exists a homeomorphism $f$ of $\Sigma_{1,1}$ fixing pointwise $\partial \Sigma_{1,1}$ such that $(f(a), f(b))$ is a special handle. We can assume moreover that no fixed point of $\rho\circ f_*(b)$ lies in a given finite set $A$ and that $\xi$ does not interchange the fixed points $\rho\circ f_*(a)$ where $\xi\in \PSL\C$.
\end{prop}

\begin{proof}
Since $\alpha=\rho(a)$ does not interchange the fixed points of $\beta=\rho(b)$, we have that $\beta^k\alpha$ is loxodromic if $k \geqslant K$ or $k\leqslant -K$ for some $K>0$. If $\xi$ does not interchange the fixed points of $\beta^k\alpha$ for such a $k$, apply the Dehn twist that changes $(a,b)$ into $(b^k a, b)$.

If however $\xi$ interchanges the fixed points of both $\alpha^k\beta$ and $\alpha^{k+1}\beta$, with $k$ and $k+1$ in that range, then \cref{echanger} shows that $\xi\alpha\xi = \alpha^{-1}$ and $\xi \beta \xi = \alpha^{-1} \beta^{-1} \alpha$. Note that we cannot have $\xi \beta^{-1} \xi = \beta$, otherwise $\alpha$ and $\beta$ would have the same fixed points. 

Thus if it is the case, we apply a sequence of Dehn twists that change the handle like this:
\[(a,b)\to (b^{-1}a, b)\to (b^{-1}, ab)\to (b^{-1}, b^{-1}a b).\]
And we return to the beginning. Therefore we may assume that $\xi$ does not interchange the fixed points of $\alpha$.

Since $\beta$ does not interchange the fixed points of $\alpha$, there exists $K >0$ such that $\alpha^m\beta$ is loxodromic for $m\geqslant K$ or for $m\leqslant -K$. Now suppose there is an infinite sequence of $m$ in that range such that $\alpha$ sends a fixed point of $\alpha^m\beta$ to the other one. Since the fixed points of $\alpha^m\beta$ tends to $p$ and $\beta^{-1}(q)$ where $p,q$ are the fixed points of $\alpha$, we have $\alpha(p) = \beta^{-1}(q)$ or $\alpha\beta^{-1}(q) = p$. Thus $p = \beta^{-1}(q)$ and $\alpha^m\beta(p) = q$ for all $m$. But if we increase $K$, we can assume $|\Tr(\alpha^m\beta)|\neq |\Tr(\alpha)|$ for $m\geqslant K$ or $m\leqslant -K$. Since $\alpha^m\beta$ sends one fixed point of $\alpha$ to the other and thanks to \cref{echanger}, $\alpha$ does not send a fixed point of $\alpha^m\beta$ to the other, a contradiction.

The fixed points of $\alpha^m\beta$ are disjoint from those of $\alpha^n \beta$ if $n\neq m$. Otherwise, $\alpha$ and $\beta$ would have a common fixed point. Hence we can also suppose that no fixed point of $\alpha^m\beta$ lies in $A$.
\end{proof}
\section{Pentagon representations}\label{Pentagon}
\subsection{Definition}

Let $\Gamma$ be the group with the following presentation: $$\Gamma = \langle q_1, q_2, q_3, q_4, q_5, q_6 \mid q_i^2 = 1, q_1q_2\ldots q_6 = 1 \rangle.$$

Define $\iota$ by $\iota(a_1) = q_2q_1$, $\iota(b_1) = q_2q_3$, $\iota(a_2) = q_5q_4$ $\iota(b_2) = q_5q_6$. 
We have $\iota([a_2, b_2][a_1, b_1]) = (q_6q_5q_4)^2(q_3q_2q_1)^2 = 1$, so $\iota : \Gamma_2\to \Gamma$ is a well-defined map.
The homomorphism $\iota$ is injective and identifies $\Gamma_2$ with an index two subgroup of $\Gamma$.

Let us recall from the introduction that a pentagon representation is the restriction of a representation $\rho : \Gamma \to \PSL\C$ that kills (\textit{i.e.} sends to the identity) exactly one $q_i$, and such that $\rho\circ \iota$ is non-elementary.
We leave it as an elementary exercise that if two or more $q_i$ are killed by $\rho$, then $\rho$ is elementary.
This property is invariant under conjugation, so it is a property of characters.

Before going on to the study of those representations, let us give an example that motivates the terminology. Consider a right-angled pentagon in the hyperbolic plane and denote its vertices by $x_1, \ldots,x_5$. Define $\rho:\Gamma\to \PSL\R$ by $\rho(q_i) = s_{x_i}$ for $i\leqslant 5$ where $s_{x_i}$ is the elliptic involution of the hyperbolic plane fixing $x_i$, and by $\rho(q_6) = \id$. This is well-defined because $s_{x_i}$ is the product of the reflection across the lines $(x_{i-1} x_i)$ and $(x_i, x_{i+1})$ (in cyclic notation) so $\rho(q_1,\ldots q_6) = \id$. The representation $\rho\circ\iota$ is non-elementary and is thus a pentagon representation.

We now show that pentagon representations have odd Stiefel-Whitney class.

\begin{prop}\label{relevement}
A pentagon representation does not lift to $\SL\C$.
\end{prop}

\begin{proof}
Let $\rho : \Gamma\to \PSL\C$ be such that $\rho\circ \iota$ is a pentagon representation.
Take $\tilde{q_i}\in \SL\C$ such that $\pm \tilde{q_i} = \rho(q_i)$ for $i\leqslant 6$.

Note that $\tilde{q_i}^{-1} = -\tilde{q_i}$ if $\rho(q_i)\neq \id$ because $\rho(q_i)$ is conjugate to $\pm \begin{pmatrix}
0 & -1\\
1 & 0
\end{pmatrix}$. Of course if $\rho(q_i) = \id$, then $\tilde{q_i}^2 = \Id$.

Now $\tilde{q_1}\tilde{q_2}\ldots\tilde{q_6}$ is a lift of $\rho(q_1q_2\ldots q_6) = \id$, so $\tilde{q_1}\tilde{q_2}\ldots\tilde{q_6} = \epsilon \Id$, with $\epsilon = \pm 1$.

Let $\tilde{a_1} = \tilde{q_2}\tilde{q_1}$, $\tilde{b_1} = \tilde{q_2}\tilde{q_3}$, $\tilde{a_2} = \tilde{q_5}\tilde{q_4}$ and $\tilde{b_2} = \tilde{q_5}\tilde{q_6}$.
\begin{align*}
  [\tilde {a_2}, \tilde {b_2}][\tilde {a_1}, \tilde {b_1}] 	& = \tilde{q_6}^{-1}\tilde{q_5}^{-1}\tilde{q_4}^{-1}\tilde{q_6}\tilde{q_5}\tilde{q_4}\tilde{q_3}^{-1}\tilde{q_2}^{-1}\tilde{q_1}^{-1}\tilde{q_3}\tilde{q_2}\tilde{q_1} \\
 			& = -(\tilde{q_6}\tilde{q_5}\tilde{q_4})^2(\tilde{q_3}\tilde{q_2}\tilde{q_1})^2 \\
 			& =- (\tilde{q_6}\tilde{q_5}\tilde{q_4})\epsilon\tilde{q_3}\tilde{q_2}\tilde{q_1} \\
 			& = - \epsilon^2 \Id = - \Id. 
\end{align*}
Therefore $\rho\circ\iota$ does not lift to $\SL\C$.
\end{proof}
\subsection{Action of the mapping class group}

The mapping class group of $\Sigma_2$ acts naturally on $\mathrm{Hom}(\Gamma_g, \PSL \C)/\PSL \C$ as follows: $[f]\cdot[\rho]=  [\rho\circ f_*^{-1}]$ where $f_*$ is the outer automorphism of $\Gamma_2$ induced by $f$.

\begin{prop}
This action preserves pentagon representations.
\end{prop}

\begin{figure}[h]
\begin{center}
\def\svgwidth{0.9\textwidth}
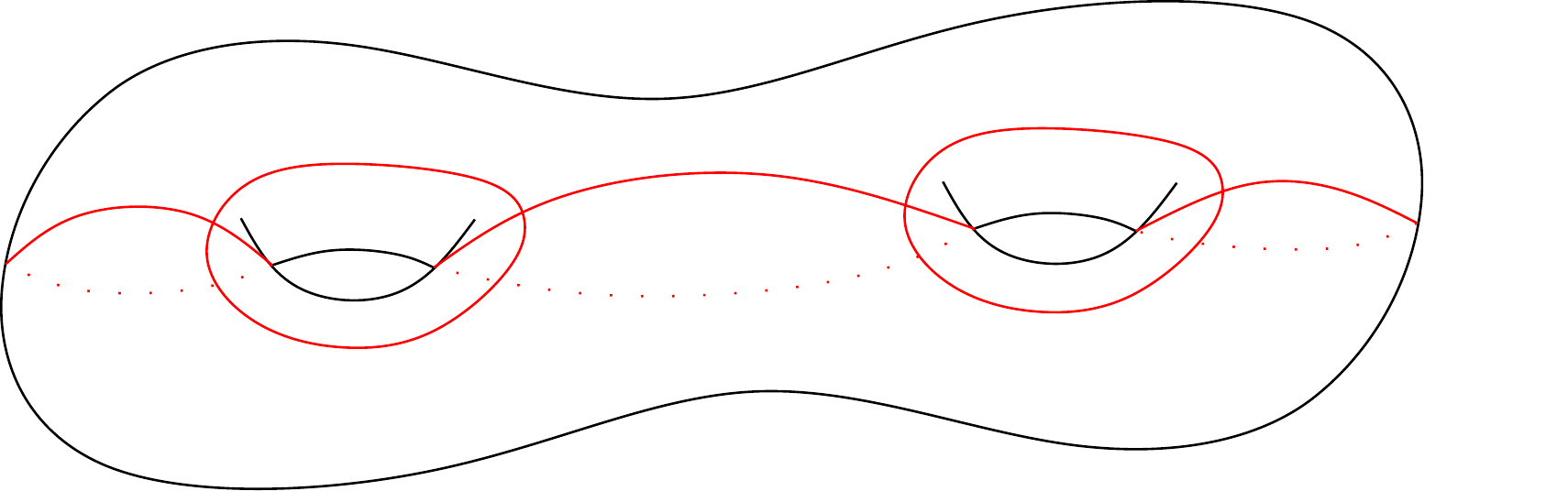
\caption{Generators of $\mathrm{Mod}(\Sigma_2)$.}
\label{gen_mod}
\end{center}\end{figure}
\begin{proof}
For $1\leqslant i\leqslant 5$, let $\sigma_i$ be the automorphism of $\Gamma$ defined by $\sigma_i(q_i) = q_{i+1}$, $\sigma(q_{i+1}) = q_{i+1} q_i q_{i+1}$ and $\sigma_i(q_j) = q_j$ for $j\neq i, i+1$.\
The mapping class group is generated by the Lickorish generators : the Dehn twists along the curves $c_i$ drawn in \cref{gen_mod}; see \cite[Chapter 4]{FarbMargalit}. 

The outer automorphism $[\varphi_i]$ of $\Gamma_2$ induced by a Dehn twist along the curve $c_i$ has a representative $\varphi_i\in \Autp{\Gamma_2}$ such that the following diagram commutes:  

\[ \begin{tikzcd}
\Gamma_2 \arrow{r}{\varphi_i} \arrow[swap]{d}{\iota} & \Gamma_2 \arrow{d}{\iota} \\%
\Gamma \arrow{r}{\sigma_i}& \Gamma.
\end{tikzcd}
\]

Hence $[\rho\circ \iota\circ \varphi_i] = [\rho\circ \sigma_i \circ\iota]$ and $\rho\circ \sigma_i$ kills exactly one $q_i$ if $\rho$ does.
\end{proof}

\section{Schottky decomposition}\label{section_schottky}
The aim of this section is to prove \cref{theo_un}: the absence of a Schottky decomposition characterizes the pentagon representations among the non-elementary representations.

\subsection{Pentagons are not Schottky}
Let us first show that a pentagon representation does not admit a Schottky decomposition.

\begin{proof}
Let $\rho : \Gamma\to \PSL\C$ be such that $\rho\circ \iota$ is a pentagon representation.
Let us consider the two pants decomposition of $\Sigma_2$ as shown in \cref{pants_dec}. The first one is not a Schottky decomposition for $\rho$. Indeed since there is a $q_i$ killed by $\rho$, its restriction to one of the two handles must be elementary. 
The second pants decomposition is not a Schottky decomposition for $\rho$ either since one can check that the image of one of the boundary curves has order $2$.

Now if there is a Schottky decomposition $P$ for $\rho$, there exists a positive homeomorphism $f$ of $\Sigma$ taking one of these two pants decompositions to $P$. The pentagon representation $[\rho\circ f_*]$ admits a Schottky decomposition with one of those two pants decompositions, which is a contradiction.
\end{proof}

\begin{figure}[h]
\begin{subfigure}{.5\textwidth}
\centering
\includegraphics[scale=0.37]{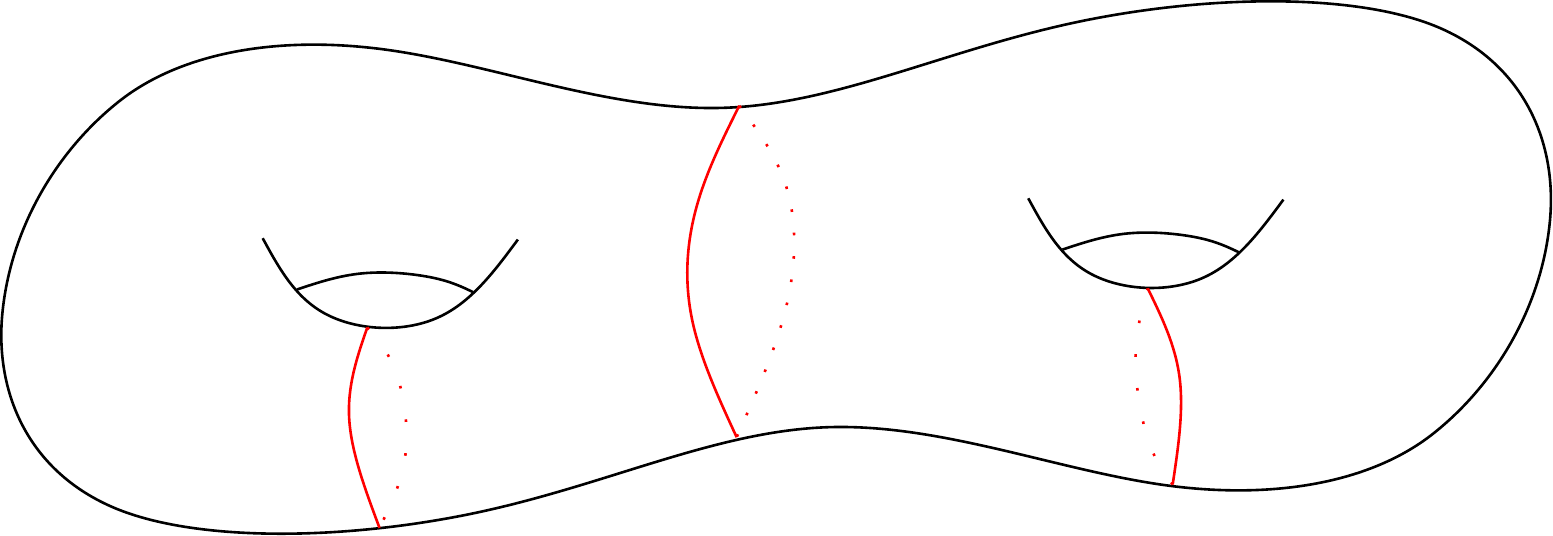}
\end{subfigure}
\begin{subfigure}{.5\textwidth}
\centering
\includegraphics[scale=0.37]{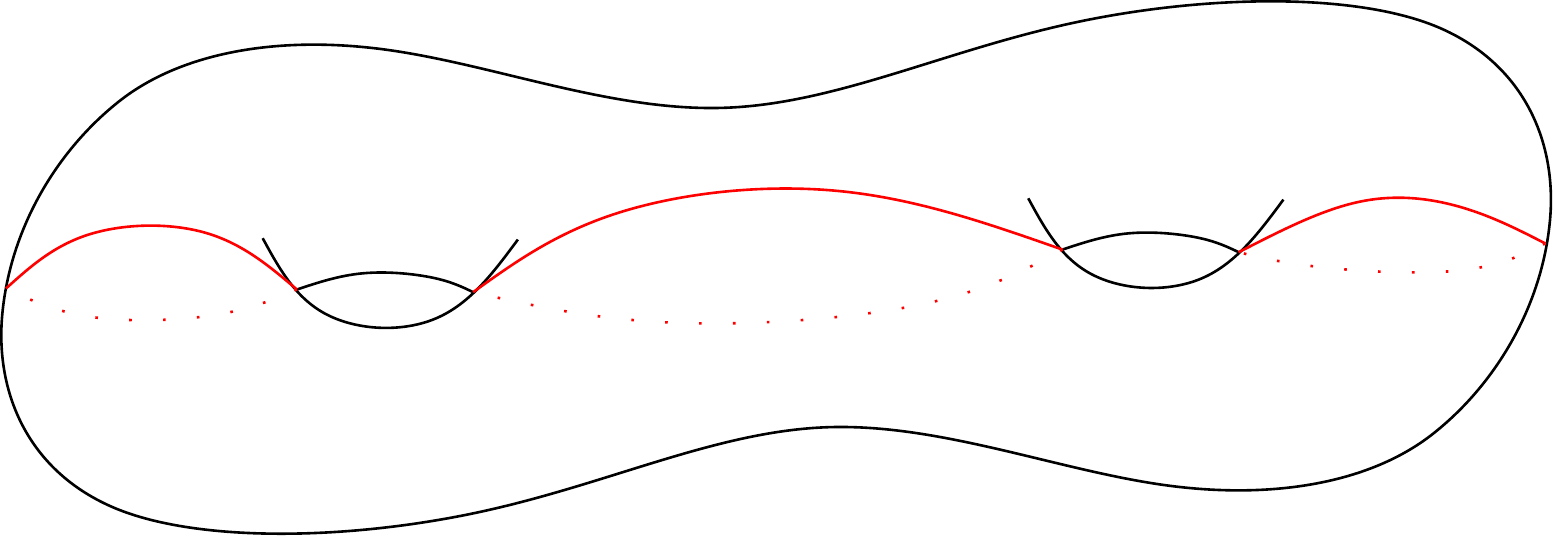}
\label{pants_dec}
\end{subfigure}

\caption{Pants decompositions of $\Sigma_2$.}
\end{figure}
\subsection{Non-Schottky are pentagons}

\subsubsection{Tools to form a Schottky decomposition}
The following proposition is a rephrasing of the paragraph 4.4 of \cite{GKM}:
\begin{prop}\label{loxdromant}
Suppose that $(a_1,b_1)$ is a special handle, and that $\rho(b_2a_1^{-1})$ does not exchange the fixed points of $\rho(b_1)$. Suppose moreover that $\rho(b_2)\neq \id$ or that $\rho(a_2)$ does not interchange the fixed points of $\rho(a_1)$. A Dehn twist of order $n$, along a curve $d_k$, freely homotopic to $b_2a_1^{-1}b_1^k$ transforms $(a,b)$ in a non-elementary handle, and $\rho(a_2)$ (or $\rho(b_2a_2)$) in a loxodromic isometry for some $k,n$.
\end{prop}

\begin{figure}[h]
\begin{center}
\def\svgwidth{0.8\textwidth}
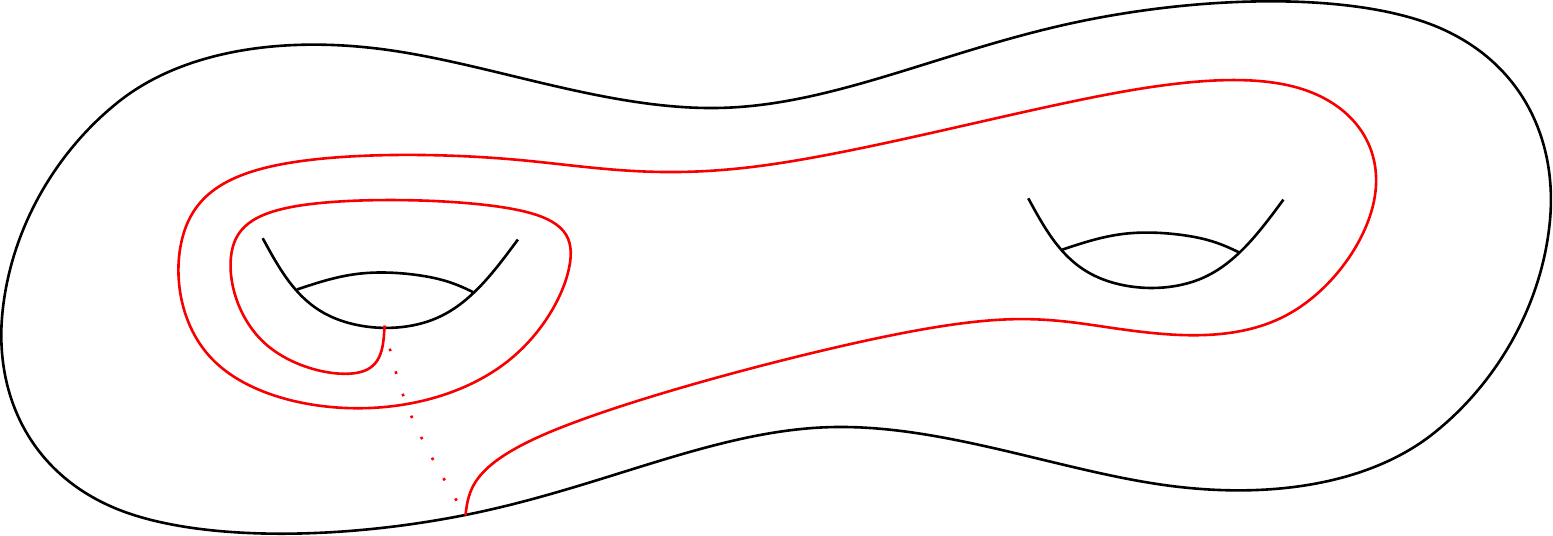
\caption{Dehn Twist.}
\end{center}
\end{figure}

The proof is very similar to the one of \cite{GKM}; we slightly simplify its beginning and modify its end.

\begin{proof}
There exists $K > 0$ such that $\rho(a_1^{-1}b_1^k)$ is loxodromic for $|k| \geqslant K$. 
The isometry $\delta_k = \rho(b_2a_1^{-1}b_1^k)$ is also loxodromic for $k\geqslant K$ or $k\leqslant -K$, increasing $K$ if necessary.

Fix such a $k$ so that $k+1$ is also in that range, and let $\delta = \delta_k$.
There is at most one $n$ such that $\rho(b_1)\delta^n$ shares a given fixed point with $\rho(a_1^{-1}b_1^k)$. For if $\rho(b_1)\delta^n(p) = \rho(a_1^{-1}b_1^k)(p) = p = \rho(b_1)\delta^{n+m}(p)$, then $\delta(p) = p = \rho(b_1)(p) = \rho(a_1)(p)$, which is a contradiction since $(a_1, b_1)$ is a special handle.

Thus there exists $N(k)\geqslant 0$ such that $\rho(b_1)\delta^n$ and $\rho(a_1^{-1}b_1^k)$ do not share a fixed point for $|n|\geqslant N(k)$: the image of the handle is non-elementary.

\begin{lemma}
It is not possible for both $\rho(a_2)$ and $\rho(b_2a_2)$ to interchange the fixed points of $\delta_k$ and of $\delta_{k+1}$.
\end{lemma}
\begin{proof}
We proceed by contradiction.
\begin{itemize}
\item If $\rho(b_2) = \id$, then $\rho(a_2)$ interchanges the fixed points of $\rho(a_1^{-1}b_1^k)$ and of $\rho(a_1^{-1}b_1^{k+1})$. By \cref{echanger}, $\rho(a_2)$ interchanges the fixed points of $\rho(a_1)$.

\item If $\rho(b_2)\neq \id$, then $\rho(b_2)$ fixes the fixed points of $\delta_k$ and of $\delta_{k+1}$, and has only two fixed points in $\partial \mathbb H^2$. Hence $\delta_k$ and $\delta_{k+1}$ share a fixed point $p$ with $\rho(b_2)$.
This implies that $\rho(b_2a_1^{-1}b_1^k)(p) = \rho(a_1^{-1}b_1^{k})(p) = p$ and $\rho(a_1^{-1}b_1^{k+1})(p) = p$. Hence both $\rho(b_1)$ and $\rho(a_1)$ fix $p$. This is a contradiction since $(a_1, b_1)$ is a special handle.~\qedhere
\end{itemize}
\end{proof}

It follows that $\rho(a_2)\delta^n $, or $\rho(b_2a_2)\delta^n$ is loxodromic for some $n \geqslant N(k)$ or $n\leqslant -N(k)$, increasing $N(k)$ is necessary. 
Note that its trace can be made arbitrarily large.
\end{proof}

Let us recall a tool from \cite{GKM} to construct Schottky groups.
\begin{lemma}\label{Schottky}
Suppose $\alpha$, $\beta$ and $\delta$ are loxodromic, and that neither $\alpha$ nor $\beta$ shares a fixed point with $\delta$.

Then $\delta^n \alpha \delta^{-n}$ and $\beta$ generate a Schottky group for $|n|$ large enough.
\end{lemma}

The following proposition explains how we can construct a Schottky decomposition, following \cite{GKM}.

\begin{prop}\label{couper_cest_gagner}
Suppose we can cut the surface along curves such that we get a surface of genus $1$, with a special handle, and such that the boundary curves are loxodromic with pairwise different images. Then there exists a Schottky pants decomposition as desired.
\end{prop}

\begin{proof}
Suppose we want to construct a Schottky pair of pants from the boundaries $d_1$ and $d_2$, such that $\rho(d_1)\neq \rho(d_2)$.
By \cref{arranger_anse}, we may assume that $\rho(b_1)$ does not fix the fixed points of $\rho(d_1d_2^{-1})$. Then, it is not possible for both $\rho(d_1a_1^{-1})$ and $\rho(d_2a_1^{-1})$ to interchange the fixed points of $\rho(b_1)$.

We can now apply the arguments of 5.2 and 5.3 of \cite{GKM} to form a Schottky pair of pants from those boundary components, and remove it from the surface. We make sure the trace of the new boundary is larger than the others'.
Following \cite{GKM}, we produce a Schottky pants decomposition.
\end{proof}

The mistake of \cite{GKM} lies in its paragraph 5.5, where a Schottky pair of pants is found, but the non-elementary handle is not always kept. We avoid using this part of their proof.

\subsubsection{The genus 2 case}
Let $\rho : \Gamma_2\to \PSL\C$ be a non-elementary representation that does not admit a Schottky decomposition.

The following proposition is an adaptation of \cite[Paragraph 4.5]{GKM}.
\begin{prop}\label{special_form}
We can change $\rho$ by some $\rho_1$ such that $[\rho_1] = f\cdot[\rho]$ with $f~\in~\homeo{\Sigma_2}$, so that $(a_1, b_1)$ is a special handle, and $\rho(b_2)$ is loxodromic.
\end{prop}

Let us explain why this result, combined with \cref{theo_un}, implies \cref{cor}.
By \cref{theo_un}, we just have to consider the pentagon representations. If a pentagon representation $\rho$ is  of the form above, it admits an extension $\rho':\Gamma\to \PSL \C$ such that $\rho'(q_4) = \id$. But then $\rho(a_1)$, $\rho([a_2, b_2])=\rho'((q_5q_6)^2)=~\rho(b_2)^2$ and $\rho(b_2)$ are loxodromic. We thus consider a pants decomposition defined by curves freely homotopic to $a_1$, $[a_2, b_2]$ and $b_2$.

\begin{proof}
We can assume $(a_1, b_1)$ is a special handle.
Let us start by changing the handle $(a_2, b_2)$ by $(a_2, a_2b_2)$ in the case where $\rho(a_2) = \rho(b_2)$ and they are of order 2. Return to the notation $(a_2, b_2)$.

If $\rho(b_2) = \id$, apply \cref{arranger_anse} to make sure that $\rho(a_2)$ does not interchange the fixed points of $\rho(a_1)$. Then we can apply \cref{loxdromant} to turn $\rho(a_2)$ into a loxodromic isometry, because $\rho(b_2a_1^{-1})$ is loxodromic and hence cannot interchange two points.

If however $\rho(b_2) \neq \id$, apply \cref{arranger_anse} to make sure that $\rho(b_2a_2)$ does not interchange the fixed points of $\rho(a_1)$. 

If $\rho(b_2a_1^{-1})$ does not interchange the fixed points of $\rho(b_1)$, we can apply \cref{loxdromant} to make $\rho(a_2)$ loxodromic. 

If it does, then suppose that $\rho(a_2^{-1}a_1^{-1})$ does not interchange the fixed points of $\rho(b_1)$. Then we modify the handle $(a_2, b_2)$ by a homeomorphism to $(b_2a_2, a_2^{-1})$. It is a composition of Dehn twists in the handle that changes it as follows:
\[(a_2, b_2) \to (a_2, a_2^{-1}b_2)\to (b_2a_2, a_2^{-1}).\] 
We can apply \cref{loxdromant} since we made sure that $\rho(b_2 a_2)$ does not interchange the fixed points of $\rho(a_1)$.

Finally if both $\rho(a_2^{-1}a_1^{-1})$ and $\rho(b_2a_1^{-1})$ interchange the fixed points of $\rho(b_1)$, then $\rho(a_2^{\pm 1} b_2a_1^{-1})$ does not, for it would imply that $\rho(a_2)$ fixes them and then $\rho(a_1)$ would interchange them. We can make sure that $\rho(a_2^{\pm 1} b_2)\neq \id$, because we are not in the case where $\rho(b_2) = \rho(a_2)$ is of order 2. We then apply a Dehn twist that does:\[(a_2, b_2) \to (a_2, a_2^{\pm 1} b_2).\]
We can then use~\cref{loxdromant}.

We have made $\rho(a_2)$ loxodromic. But again we can change the handle $(a_2, b_2)$ as before to make sure $\rho(a_2^{-1})$ is loxodromic. The handle $(a_1, b_1)$ is non-elementary and we can improve it to a special handle with~\cref{arranger_anse}.
\end{proof}

We now just have to consider pentagon representations in this special form.

\begin{prop}
The homomorphism $\rho$ is a pentagon representation.
\end{prop}

\begin{proof}
The axes of $\rho(a_1)$ and of $\rho(b_1)$ do not cross in $\partial \mathbb H^3$, so there exists a unique line $\ell$ orthogonal to both of them. Let $q_2 = s_\ell$ be the elliptic involution with axis $\ell$. Then $q_1 = q_2\rho(a_1)$ and $q_3 = q_2\rho(b_1)$ are elliptic involutions.

We have $\rho(a_2)^{-1}\rho(b_2)^{-1} \rho(a_2) = \rho(b_2)$. Indeed, otherwise we could cut the surface along a curve freely homotopic to $b_2$, and use \cref{couper_cest_gagner}. The isometry $q_5 = \rho(a_2)$ interchanges the fixed points of $\rho(b_2)$, and is an elliptic involution with an axis orthogonal to the one of $\rho(b_2)$. Hence $q_6 = q_5 \rho(b_2)$ has order 2.

We have $\id = \rho([a_2 b_2][a_1, b_1]) = (q_6q_5)^2 (q_3q_2q_1)^2$. This implies that $q_1q_2q_3 = q_6q_5\circ r$ where $r$ commutes with $q_6q_5$ and is such that $r^2= \id$ because of the following lemma.

\begin{lemma}
If $f$ and $g$ are loxodromic isometries such that $f^2 = g^2$, either $f=g$ or $f=g\circ r$ where $r$ is an elliptic involution with the same axis as $f$ and~$g$.
\end{lemma}

\begin{proof}
They have the same axis so after conjugating we can write $f(z) = \lambda z$ and $g(z) = \mu z$. We have $\lambda^2 = \mu^2$; hence $\lambda = \mu$ or $\lambda = -\mu$.
\end{proof}

Thanks to \cref{couper_cest_gagner},  $\rho(b_2a_1^{-1})$ must interchange the fixed points of $\rho(b_1)$. Indeed otherwise we could apply the \cref{loxdromant} to improve the situtation and suppose that $\rho(a_2)$ is loxodromic. Then $\rho(a_2)^{-1}\rho(b_2)^{-1}\rho(a_2) = \rho(b_2)$ would be impossible.

We have $\rho(b_2a_1^{-1}) = q_5q_6q_1q_2 = rq_3$ and $\rho(b_1) = q_2q_3$. Therefore $rq_3$ is of order 2 and $rq_3 = (rq_3)^{-1} = q_3 r$. Moreover,
$(rq_3) q_2q_3 (rq_3) = q_3q_2$, thus $rq_3q_2 r = q_3q_2$. The centralizer of $r$ contains $q_1q_2q_3$, $q_3$ and $q_3q_2$, hence a non-elementary group. This implies that $r = \id$.
\end{proof}

\subsubsection{Genus $g\geqslant 3$}
In this subsection $g\geqslant 3$ and $\rho: \Gamma_g\to \PSL \C$ is non-elementary.

\begin{prop}
There exists a Schottky decomposition for $\rho$.
\end{prop}

\begin{proof}
Thanks to \cref{couper_cest_gagner}, it suffices to show that we can cut the surface along non-separating curves having different loxodromic images.

We can assume that $(a_1, b_1)$ is a special handle. We can also suppose that each $\rho(b_i)$ is loxodromic, by applying successively  the arguments of \cref{special_form}. We can moreover suppose that the traces of the $\rho(b_i)$ are pairwise distinct.

Use \cref{arranger_anse} to make sure that $\rho(b_2)$ does not fix the fixed points of $\rho(b_1)$.
It is not possible for both $\rho(b_3 a_1^{-1})$ and $\rho(b_2b_3 a_1^{-1})$ to interchange the fixed points of $\rho(b_1)$, because that would mean that $\rho(b_2)$ fixes them. 
We now apply the \cref{loxdromant} in the handle $(a_3, b_3)$ or $(a_3, b_2b_3)$.
Note that $\rho(b_3)\neq \id$ and $\rho(b_2b_3)\neq \id$ for the trace of $\rho(b_2)$ is different from the trace of $\rho(b_3)$. This changes $\rho(a_3)$ into a loxodromic isometry and leaves $\rho(b_3)$ unchanged.
Thus both $\rho(a_3)$ and $\rho(b_3)$ are loxodromic; we cut the surface along a curve freely homotopic to $a_3$. We cannot have $\rho(a_3) = \rho(b_3)^{-1}\rho(a_3)^{-1}\rho(b_3)$ for it would imply that $\rho(b_3)$ interchanges the fixed points of $\rho(a_3)$.

We can repeat the argument with the other handles while there are at least two handles to cut. Note that the images of the boundary components are modified by conjugation at each step.

We are left with a special handle, $2(g-2)$ boundary components, and a handle that we want to cut.
We may assume that $(a_1, b_1)$ is a special handle, and that the handle we whish to cut is $(a_2, b_2)$, and that $\rho(b_2)$ is loxodromic. Make sure as before that $\rho(b_3)$ does not fix the fixed points of $\rho(b_1)$.

If $\rho(b_2a_1^{-1})$ does not interchange the fixed points of $\rho(b_1)$, then we apply \cref{loxdromant} in the handle $(a_2, b_2)$.
Otherwise, we apply the same proposition in the handle $(a_2, b_2a_3^{-1})$. The boundary components $d_1 = b_3$ and $d_2 = b_3^{-1}a_3^{-1}b_3$ corresponding to the handle $(a_3, b_3)$ are changed as follows: $(d_1, d_2)\to (d_1, \zeta^{-n}d_2\zeta^n)$ where $\zeta = a_1^{-1}b_1^kb_2b_3$, and $k$ and $n$ come from the \cref{loxdromant}.
We cannot have $d_1 = \zeta^{-n}d_2\zeta^n = \zeta^{-n-1}d_2\zeta^{n+1}$ for it would imply $d_1=d_2$. Thus we can assume that $d_1\neq d_2$ after this Dehn twist is done.

We can thus cut again, to have a genus $1$ surface with a special handle, and $2(g-1)$ boundary components, with any two of them having different images.
\end{proof}

\section{Projective structure}\label{section_projective}
The goal of this section is to prove \cref{PentagonAreProjective}.
As before, we can modify $\rho$ by the action of the mapping class group since the property we are interested in is invariant under this action.

Even if the non-elementary handle is not kept, we still find a Schottky pair of pants following \cite[Section 5.5]{GKM}.

\begin{prop}\label{forme_schottky}
There exists $\rho_1$ such that $[\rho_1] = f\cdot [\rho]$ for some $f\in \homeo{\Sigma_2}$, such that $\rho_1(a_2)$ and $\rho_1(b_2^{-1} a_2^{-1} b_2)$ generate a Schottky group, and such that $\rho_1(b_1)$ is loxodromic and $\rho_1(a_1)$ interchanges the fixed points of $\rho_1(b_1)$.
\end{prop}

\begin{proof}
We may assume that $\rho$ is as in \cref{special_form}. Denote by $\rho'$ an extension of $\rho$ to $\Gamma$. Cut $\Sigma_2$ along a curve freely homotopic to $b_2$, so that we get a genus $1$ surface with $2$ boundary components having loxodromic images : $d_1 = b_2$ and $d_2 = a_2^{-1}b_2^{-1}a_2$. Their images are equal : $\rho'(q_4) = \id$ and $\rho'(q_5 q_6q_5 q_5) = \rho'(q_5q_6)$. The isometries $\rho((b_1^{-k}a_1)d_1(a_1^{-1}b_1^k))$ and $\rho(d_2)$ generate a Schottky group for $|k|$ large enough. Indeed $\rho(d_2) = \rho'(q_5q_6)$ cannot fix a fixed point of $\rho(b_1) = \rho'(q_2q_3)$ for $\rho(a_1^{-1}b_2) = \rho'(q_1q_2q_5q_6) = \rho'(q_3)$ interchanges those of $\rho(b_1)$ and thus $\rho(a_1)$ would send a fixed point of $\rho(b_1)$ to the other. Similarly, $\rho(a_1^{-1} a_1 d_1 a_1^{-1})$ interchanges the fixed points of $\rho(b_1)$, hence $\rho(a_1 d_1 a_1^{-1})$ does not fix any of them. We thus get a Schottky pair of pants by \cref{Schottky}.

We may assume that the Schottky pair of pants comes from cutting the handle $(a_2, b_2)$. One of $q_1$, $q_2$ and $q_3$ is killed by $\rho'$. We can assume that it is $q_1$, applying a homeomorphism of the handle $(a_1, b_1)$ if necessary. Since $\rho(a_1)^{-1} \rho(b_1)\rho(a_1) = \rho(b_1)^{-1}$, the map $\rho(b_1)^2 = \rho([a_1, b_1])^{-1} = \rho([a_2, b_2])$ is loxodromic, and so is $\rho(b_1)$.
\end{proof}

We now put a projective structure on $\Sigma_{1,1}$ whose holonomy is the non-Schottky part of the previous proposition. Namely, this holonomy is given by $\rho$ which maps the standard generators $a_1,b_1$ of $\Gamma_{1,1}$ to $\rho(a_1)$, which is loxodromic, and to $\rho(b_1)$ which is an involution interchanging the fixed points of $\rho(a_1)$.

\begin{prop}
There exists a projective structure with a single branched point of order 2 on $\Sigma_{1,1}$ such that its holonomy is $\rho$, and such that the developing map embeds the boundary curve in $\mathbb {CP}^1$.
\end{prop}

\begin{figure}[h]
\begin{center}
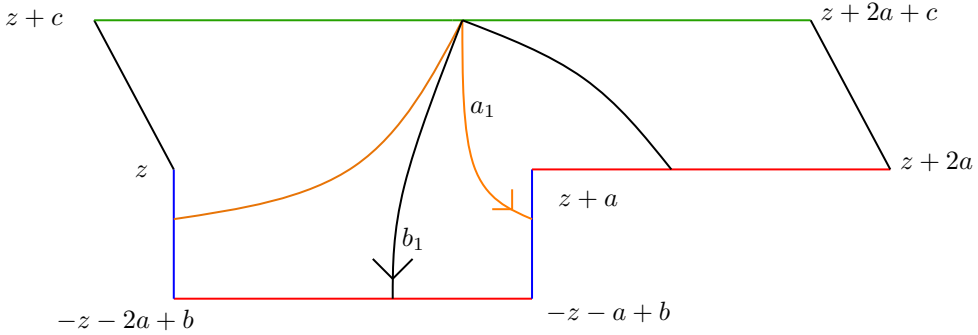
\caption{Affine structure on $\Sigma_{1,1}$.}
\label{polygone}
\end{center}
\end{figure}

\begin{proof}
Given $\mu\in \C\setminus \{0\}$, we can conjugate $\rho$, so that $\rho(a_1) : z\mapsto \lambda^{-2}z$ and $\rho(b_1) : z\mapsto \mu z^{-1}$ for some $|\lambda| > 1$.

Let $z\in \C$ and $a\in \C$ be such that $e^a = \lambda^2$. Construct the polygon $P$ of \cref{polygone} where $b$ and $c$ are complex numbers so that it is an actual polygon (without self-crossing). 

The identifications of the sides by $z\mapsto z+a$ for the blue ones, $z\mapsto -z+b$ for the red ones, and $z\mapsto z + 2a$ for the black ones give an affine structure on $S$, with a cone point of angle $4\pi$. Taking the exponential of small enough charts defines a complex projective structure on $S$ with a single branched point of order $2$.

The holonomy of this projective structure maps $a_1$ to $z\mapsto \lambda^{-2}z$. Indeed, if $z = e^\omega$, then $e^{\omega-a} = \lambda^{-2}z$. Similarly, it maps $b_1$ to $z\mapsto e^bz^{-1}$.

Since $\Re(a) > 0$, the developing map embeds the boundary curve in $\mathbb{CP}^1$.
\end{proof}

We are now able to prove \cref{PentagonAreProjective}. We are reduced to the case where $\rho$ is of the form of \cref{forme_schottky}. Put the branched projective structure as above on the handle that is not Schottky. We can put a projective structure on the Schottky handle with the desired holonomy that is compatible (\textit{i.e.} that we can glue to the other one), possibly with a branched point of order $2$ (see \cite[Paragraph~7,8,9]{GKM}).

But it is not possible for $\rho$ to be the holonomy of a branched projective structure with two branched points of order 2, since it would imply that it lifts to $\SL \C$ (see \cite[Corollary~11.2.3]{GKM}), contradicting \cref{relevement}.

\bibliographystyle{plain}
\bibliography{bibpapier}

\end{document}

%% file: generateurs.pdf_tex
\begingroup%
  \makeatletter%
  \providecommand\color[2][]{%
    \errmessage{(Inkscape) Color is used for the text in Inkscape, but the package 'color.sty' is not loaded}%
    \renewcommand\color[2][]{}%
  }%
  \providecommand\transparent[1]{%
    \errmessage{(Inkscape) Transparency is used (non-zero) for the text in Inkscape, but the package 'transparent.sty' is not loaded}%
    \renewcommand\transparent[1]{}%
  }%
  \providecommand\rotatebox[2]{#2}%
  \newcommand*\fsize{\dimexpr\f@size pt\relax}%
  \newcommand*\lineheight[1]{\fontsize{\fsize}{#1\fsize}\selectfont}%
  \ifx\svgwidth\undefined%
    \setlength{\unitlength}{447.04482446bp}%
    \ifx\svgscale\undefined%
      \relax%
    \else%
      \setlength{\unitlength}{\unitlength * \real{\svgscale}}%
    \fi%
  \else%
    \setlength{\unitlength}{\svgwidth}%
  \fi%
  \global\let\svgwidth\undefined%
  \global\let\svgscale\undefined%
  \makeatother%
  \begin{picture}(1,0.38173117)%
    \lineheight{1}%
    \setlength\tabcolsep{0pt}%
    \put(0,0){\includegraphics[width=\unitlength,page=1]{generateurs.pdf}}%
    \put(0.27830602,0.27529253){\color[rgb]{0,0,0}\makebox(0,0)[lt]{\lineheight{1.25}\smash{\begin{tabular}[t]{l}$b_1$\\\end{tabular}}}}%
    \put(0.30976263,0.18002404){\color[rgb]{0,0,0}\makebox(0,0)[lt]{\lineheight{1.25}\smash{\begin{tabular}[t]{l}$a_1$\end{tabular}}}}%
    \put(0.47183891,0.12400137){\color[rgb]{0,0,0}\makebox(0,0)[lt]{\lineheight{1.25}\smash{\begin{tabular}[t]{l}$c_1$\end{tabular}}}}%
    \put(0.68163049,0.33357184){\color[rgb]{0,0,0}\makebox(0,0)[lt]{\lineheight{1.25}\smash{\begin{tabular}[t]{l}$a_2$\end{tabular}}}}%
    \put(1.99128631,0.46189405){\color[rgb]{0,0,0}\makebox(0,0)[lt]{\begin{minipage}{0.88552296\unitlength}\raggedright \end{minipage}}}%
    \put(0.73600682,0.12986728){\color[rgb]{0,0,0}\makebox(0,0)[lt]{\lineheight{1.25}\smash{\begin{tabular}[t]{l}$b_2$\end{tabular}}}}%
  \end{picture}%
\endgroup%

%% file: inter.pdf_tex
\begingroup%
  \makeatletter%
  \providecommand\color[2][]{%
    \errmessage{(Inkscape) Color is used for the text in Inkscape, but the package 'color.sty' is not loaded}%
    \renewcommand\color[2][]{}%
  }%
  \providecommand\transparent[1]{%
    \errmessage{(Inkscape) Transparency is used (non-zero) for the text in Inkscape, but the package 'transparent.sty' is not loaded}%
    \renewcommand\transparent[1]{}%
  }%
  \providecommand\rotatebox[2]{#2}%
  \newcommand*\fsize{\dimexpr\f@size pt\relax}%
  \newcommand*\lineheight[1]{\fontsize{\fsize}{#1\fsize}\selectfont}%
  \ifx\svgwidth\undefined%
    \setlength{\unitlength}{112.48242197bp}%
    \ifx\svgscale\undefined%
      \relax%
    \else%
      \setlength{\unitlength}{\unitlength * \real{\svgscale}}%
    \fi%
  \else%
    \setlength{\unitlength}{\svgwidth}%
  \fi%
  \global\let\svgwidth\undefined%
  \global\let\svgscale\undefined%
  \makeatother%
  \begin{picture}(1,0.93347918)%
    \lineheight{1}%
    \setlength\tabcolsep{0pt}%
    \put(0,0){\includegraphics[width=\unitlength,page=1]{inter.pdf}}%
    \put(0.10001563,0.56675521){\color[rgb]{0,0,0}\makebox(0,0)[lt]{\lineheight{1.25}\smash{\begin{tabular}[t]{l}$\gamma$\end{tabular}}}}%
    \put(0.56675524,0.13335423){\color[rgb]{0,0,0}\makebox(0,0)[lt]{\lineheight{1.25}\smash{\begin{tabular}[t]{l}$\delta$\end{tabular}}}}%
  \end{picture}%
\endgroup%

%% file: trianglebis.pdf_tex
\begingroup%
  \makeatletter%
  \providecommand\color[2][]{%
    \errmessage{(Inkscape) Color is used for the text in Inkscape, but the package 'color.sty' is not loaded}%
    \renewcommand\color[2][]{}%
  }%
  \providecommand\transparent[1]{%
    \errmessage{(Inkscape) Transparency is used (non-zero) for the text in Inkscape, but the package 'transparent.sty' is not loaded}%
    \renewcommand\transparent[1]{}%
  }%
  \providecommand\rotatebox[2]{#2}%
  \newcommand*\fsize{\dimexpr\f@size pt\relax}%
  \newcommand*\lineheight[1]{\fontsize{\fsize}{#1\fsize}\selectfont}%
  \ifx\svgwidth\undefined%
    \setlength{\unitlength}{264.90843172bp}%
    \ifx\svgscale\undefined%
      \relax%
    \else%
      \setlength{\unitlength}{\unitlength * \real{\svgscale}}%
    \fi%
  \else%
    \setlength{\unitlength}{\svgwidth}%
  \fi%
  \global\let\svgwidth\undefined%
  \global\let\svgscale\undefined%
  \makeatother%
  \begin{picture}(1,1.02427705)%
    \lineheight{1}%
    \setlength\tabcolsep{0pt}%
    \put(0,0){\includegraphics[width=\unitlength,page=1]{trianglebis.pdf}}%
    \put(-0.00376014,0.2764066){\color[rgb]{0,0,0}\makebox(0,0)[lt]{\lineheight{1.25}\smash{\begin{tabular}[t]{l}$\ell'$\end{tabular}}}}%
    \put(0.02596711,0.7562894){\color[rgb]{0,0,0}\makebox(0,0)[lt]{\lineheight{1.25}\smash{\begin{tabular}[t]{l}$\ell_n$\end{tabular}}}}%
    \put(0.35721364,0.98986067){\color[rgb]{0,0,0}\makebox(0,0)[lt]{\lineheight{1.25}\smash{\begin{tabular}[t]{l}$\ell''$\end{tabular}}}}%
  \end{picture}%
\endgroup%

%% file: triangle.pdf_tex
\begingroup%
  \makeatletter%
  \providecommand\color[2][]{%
    \errmessage{(Inkscape) Color is used for the text in Inkscape, but the package 'color.sty' is not loaded}%
    \renewcommand\color[2][]{}%
  }%
  \providecommand\transparent[1]{%
    \errmessage{(Inkscape) Transparency is used (non-zero) for the text in Inkscape, but the package 'transparent.sty' is not loaded}%
    \renewcommand\transparent[1]{}%
  }%
  \providecommand\rotatebox[2]{#2}%
  \newcommand*\fsize{\dimexpr\f@size pt\relax}%
  \newcommand*\lineheight[1]{\fontsize{\fsize}{#1\fsize}\selectfont}%
  \ifx\svgwidth\undefined%
    \setlength{\unitlength}{342.4756155bp}%
    \ifx\svgscale\undefined%
      \relax%
    \else%
      \setlength{\unitlength}{\unitlength * \real{\svgscale}}%
    \fi%
  \else%
    \setlength{\unitlength}{\svgwidth}%
  \fi%
  \global\let\svgwidth\undefined%
  \global\let\svgscale\undefined%
  \makeatother%
  \begin{picture}(1,0.94840293)%
    \lineheight{1}%
    \setlength\tabcolsep{0pt}%
    \put(0,0){\includegraphics[width=\unitlength,page=1]{triangle.pdf}}%
    \put(0.29753591,0.39240611){\color[rgb]{0,0,0}\makebox(0,0)[lt]{\lineheight{1.25}\smash{\begin{tabular}[t]{l}$\alpha$\end{tabular}}}}%
    \put(0.55114418,0.37703345){\color[rgb]{0,0,0}\makebox(0,0)[lt]{\lineheight{1.25}\smash{\begin{tabular}[t]{l}$\beta$\end{tabular}}}}%
    \put(0.11220715,0.39105571){\color[rgb]{0,0,0}\makebox(0,0)[lt]{\begin{minipage}{0.56116823\unitlength}\raggedright \end{minipage}}}%
    \put(0.42499036,0.31502424){\color[rgb]{0,0,0}\makebox(0,0)[lt]{\lineheight{1.25}\smash{\begin{tabular}[t]{l}$\ell_2$\end{tabular}}}}%
    \put(0.28707252,0.52903297){\color[rgb]{0,0,0}\makebox(0,0)[lt]{\lineheight{1.25}\smash{\begin{tabular}[t]{l}$\ell_1$\end{tabular}}}}%
    \put(0.55487491,0.52949972){\color[rgb]{0,0,0}\makebox(0,0)[lt]{\lineheight{1.25}\smash{\begin{tabular}[t]{l}$\ell_3$\end{tabular}}}}%
  \end{picture}%
\endgroup%

%% file: genmcg.pdf_tex
\begingroup%
  \makeatletter%
  \providecommand\color[2][]{%
    \errmessage{(Inkscape) Color is used for the text in Inkscape, but the package 'color.sty' is not loaded}%
    \renewcommand\color[2][]{}%
  }%
  \providecommand\transparent[1]{%
    \errmessage{(Inkscape) Transparency is used (non-zero) for the text in Inkscape, but the package 'transparent.sty' is not loaded}%
    \renewcommand\transparent[1]{}%
  }%
  \providecommand\rotatebox[2]{#2}%
  \newcommand*\fsize{\dimexpr\f@size pt\relax}%
  \newcommand*\lineheight[1]{\fontsize{\fsize}{#1\fsize}\selectfont}%
  \ifx\svgwidth\undefined%
    \setlength{\unitlength}{492.28822779bp}%
    \ifx\svgscale\undefined%
      \relax%
    \else%
      \setlength{\unitlength}{\unitlength * \real{\svgscale}}%
    \fi%
  \else%
    \setlength{\unitlength}{\svgwidth}%
  \fi%
  \global\let\svgwidth\undefined%
  \global\let\svgscale\undefined%
  \makeatother%
  \begin{picture}(1,0.3130164)%
    \lineheight{1}%
    \setlength\tabcolsep{0pt}%
    \put(0,0){\includegraphics[width=\unitlength,page=1]{genmcg.pdf}}%
    \put(0.46162847,0.21234883){\color[rgb]{0,0,0}\makebox(0,0)[lt]{\lineheight{1.25}\smash{\begin{tabular}[t]{l}$c_3$\end{tabular}}}}%
    \put(0.23463115,0.22004365){\color[rgb]{0,0,0}\makebox(0,0)[lt]{\lineheight{1.25}\smash{\begin{tabular}[t]{l}$c_2$\end{tabular}}}}%
    \put(0.09056261,0.19253741){\color[rgb]{0,0,0}\makebox(0,0)[lt]{\lineheight{1.25}\smash{\begin{tabular}[t]{l}$c_1$\end{tabular}}}}%
    \put(0.67075244,0.24255377){\color[rgb]{0,0,0}\makebox(0,0)[lt]{\lineheight{1.25}\smash{\begin{tabular}[t]{l}$c_4$\end{tabular}}}}%
    \put(0.82464899,0.21023552){\color[rgb]{0,0,0}\makebox(0,0)[lt]{\lineheight{1.25}\smash{\begin{tabular}[t]{l}$c_5$\end{tabular}}}}%
  \end{picture}%
\endgroup%

%% file: tdd.pdf_tex
\begingroup%
  \makeatletter%
  \providecommand\color[2][]{%
    \errmessage{(Inkscape) Color is used for the text in Inkscape, but the package 'color.sty' is not loaded}%
    \renewcommand\color[2][]{}%
  }%
  \providecommand\transparent[1]{%
    \errmessage{(Inkscape) Transparency is used (non-zero) for the text in Inkscape, but the package 'transparent.sty' is not loaded}%
    \renewcommand\transparent[1]{}%
  }%
  \providecommand\rotatebox[2]{#2}%
  \newcommand*\fsize{\dimexpr\f@size pt\relax}%
  \newcommand*\lineheight[1]{\fontsize{\fsize}{#1\fsize}\selectfont}%
  \ifx\svgwidth\undefined%
    \setlength{\unitlength}{447.04481595bp}%
    \ifx\svgscale\undefined%
      \relax%
    \else%
      \setlength{\unitlength}{\unitlength * \real{\svgscale}}%
    \fi%
  \else%
    \setlength{\unitlength}{\svgwidth}%
  \fi%
  \global\let\svgwidth\undefined%
  \global\let\svgscale\undefined%
  \makeatother%
  \begin{picture}(1,0.34469535)%
    \lineheight{1}%
    \setlength\tabcolsep{0pt}%
    \put(0,0){\includegraphics[width=\unitlength,page=1]{tdd.pdf}}%
    \put(0.43320118,0.15785913){\color[rgb]{0,0,0}\makebox(0,0)[lt]{\lineheight{1.25}\smash{\begin{tabular}[t]{l}$*$\end{tabular}}}}%
  \end{picture}%
\endgroup%

%% file: polygone.pdf_tex
\begingroup%
  \makeatletter%
  \providecommand\color[2][]{%
    \errmessage{(Inkscape) Color is used for the text in Inkscape, but the package 'color.sty' is not loaded}%
    \renewcommand\color[2][]{}%
  }%
  \providecommand\transparent[1]{%
    \errmessage{(Inkscape) Transparency is used (non-zero) for the text in Inkscape, but the package 'transparent.sty' is not loaded}%
    \renewcommand\transparent[1]{}%
  }%
  \providecommand\rotatebox[2]{#2}%
  \newcommand*\fsize{\dimexpr\f@size pt\relax}%
  \newcommand*\lineheight[1]{\fontsize{\fsize}{#1\fsize}\selectfont}%
  \ifx\svgwidth\undefined%
    \setlength{\unitlength}{381.87472219bp}%
    \ifx\svgscale\undefined%
      \relax%
    \else%
      \setlength{\unitlength}{\unitlength * \real{\svgscale}}%
    \fi%
  \else%
    \setlength{\unitlength}{\svgwidth}%
  \fi%
  \global\let\svgwidth\undefined%
  \global\let\svgscale\undefined%
  \makeatother%
  \begin{picture}(1,0.32990183)%
    \lineheight{1}%
    \setlength\tabcolsep{0pt}%
    \put(0,0){\includegraphics[width=\unitlength,page=1]{polygone.pdf}}%
    \put(0.38878768,0.08297332){\color[rgb]{0,0,0}\makebox(0,0)[lt]{\lineheight{1.25}\smash{\begin{tabular}[t]{l}$b_1$\end{tabular}}}}%
    \put(0.45612466,0.21414012){\color[rgb]{0,0,0}\makebox(0,0)[lt]{\lineheight{1.25}\smash{\begin{tabular}[t]{l}$a_1$\end{tabular}}}}%
    \put(0.53117732,0.01002496){\color[rgb]{0,0,0}\makebox(0,0)[lt]{\lineheight{1.25}\smash{\begin{tabular}[t]{l}$-z-a+b$\end{tabular}}}}%
    \put(2.11989233,-0.88341851){\color[rgb]{0,0,0}\makebox(0,0)[lt]{\begin{minipage}{0.03711486\unitlength}\raggedright \end{minipage}}}%
    \put(2.08277747,-0.80918879){\color[rgb]{0,0,0}\makebox(0,0)[lt]{\begin{minipage}{0\unitlength}\raggedright \end{minipage}}}%
    \put(0.04789431,0.00230923){\color[rgb]{0,0,0}\makebox(0,0)[lt]{\lineheight{1.25}\smash{\begin{tabular}[t]{l}$-z-2a+b$\end{tabular}}}}%
    \put(0.12505124,0.15171312){\color[rgb]{0,0,0}\makebox(0,0)[lt]{\lineheight{1.25}\smash{\begin{tabular}[t]{l}$z$\end{tabular}}}}%
    \put(-0.00260843,0.30111703){\color[rgb]{0,0,0}\makebox(0,0)[lt]{\lineheight{1.25}\smash{\begin{tabular}[t]{l}$z+c$\end{tabular}}}}%
    \put(0.80192797,0.30602702){\color[rgb]{0,0,0}\makebox(0,0)[lt]{\lineheight{1.25}\smash{\begin{tabular}[t]{l}$z+2a+c$\end{tabular}}}}%
    \put(0.88048784,0.15872739){\color[rgb]{0,0,0}\makebox(0,0)[lt]{\lineheight{1.25}\smash{\begin{tabular}[t]{l}$z+2a$\end{tabular}}}}%
    \put(0.54310158,0.12295464){\color[rgb]{0,0,0}\makebox(0,0)[lt]{\lineheight{1.25}\smash{\begin{tabular}[t]{l}$z+a$\end{tabular}}}}%
  \end{picture}%
\endgroup%